\documentclass[]{amsart}
\usepackage{graphicx}
\usepackage[T1]{fontenc}
\usepackage{amsmath}
\usepackage{latexsym}
\usepackage{graphicx}
\usepackage{graphics,epsfig}
\usepackage{color}
\usepackage{a4wide,mathrsfs}

\usepackage{systeme,amsfonts,amsthm,mathtools,amssymb,siunitx}
\usepackage{enumerate}
\usepackage{multicol}
\usepackage{float}
\usepackage{physics}
\usepackage[capitalize]{cleveref}
\usepackage{todonotes}
\numberwithin{equation}{section}

\newtheorem*{Proof of Theorem 1}{Proof of Theorem 1}{\it}{\rm}

\theoremstyle{remark}

\newtheorem{definition}{Definition}[section]
\newtheorem{theorem}{Theorem}[section]
\crefname{theorem}{Theorem}{Theorems}
\newtheorem{lemma}{Lemma}[section]
\crefname{lemma}{Lemma}{Lemmas}
\newtheorem{example}{Example}[section]

\crefname{corollary}{Corollary}{Corollaries}
\newtheorem{assump}{Assumption}[section]

\newcommand\Set[2]{\{\,#1\mid#2\,\}}

\newcommand{\codim}{\operatorname{\rm codim}}
\newcommand{\Ran}{\operatorname{\rm Ran}}

\DeclareMathOperator{\spn}{span}
\renewcommand{\vec}[1]{\mathbf{#1}}

\title[Perturbations of embedded eigenvalues for self-adjoint ODE systems]{Perturbations of embedded eigenvalues \\ for self-adjoint ODE systems}

\author{Sara Maad Sasane}
\address{Sara Maad Sasane, Lund University}
\email{sara@maths.lth.se}

\author{Alexia Papalazarou}  
\address{Alexia Papalazarou, Lund University}
\email{alexia.papalazarou@math.lth.se} 

\date{\today}

\begin{document}

\begin{abstract}\label{S:abstract}
We consider a perturbation problem for embedded eigenvalues of a self-adjoint differential operator in $L^2(\mathbb R;\mathbb R^n)$. In particular, we study the set of all small perturbations in an appropriate Banach space for which the embedded eigenvalue remains embedded in the continuous spectrum. We show that this set of small perturbations forms a smooth manifold and we specify its co-dimension. Our methods involve the use of exponential dichotomies, their roughness property and Lyapunov-Schmidt reduction.
\end{abstract}

\maketitle


\section{Introduction}\label{S:intro}
\label{sec:1}
\noindent
It is well known that an eigenvalue which is separated from the rest of the spectrum is stable in the sense that an added small perturbation can only move, but not remove the eigenvalue \cite[p.213]{Kato}. In contrast, eigenvalues which are embedded in the continuous spectrum behave in a completely different way. Those eigenvalues may be very unstable under perturbations, and typically, an arbitrary small perturbation makes the eigenvalue disappear.

Embedded eigenvalues occur in many applications arising in physics. For example, in quantum mechanics, eigenvalues express the energy bound states of the energy operator that can be attained by the underlying physical system. Therefore, if such an eigenvalue is embedded in the continuous spectrum it is important to determine whether it, and consequently, the bound state, persists when perturbing the potential. 
Another possibility is the occurrence of embedded eigenvalues in inverse scattering problems. In this setting, embedded eigenvalues correspond to soliton-type structures for the original integrable problems whose robustness under perturbations is therefore again determined by the fate of the embedded eigenvalue.

A different motivation for the same question arises in systems that support nonlinear waves (e.g. water waves) or vortex solutions (e.g. in photonic lattices or other nonlinear optical systems). In both cases, the main difficulties are again embedded eigenvalues which will always be present due to the specific nature of the nonlinear systems describing them. For instance, it is often of interest to construct other types of waves from a given nonlinear wave: starting with a water wave, for example, with a localized profile in one spatial direction, one may wish to glue several well-separated copies of the original water wave together to yield a wave with several elevated humps. The potential appearing in the energy operator corresponding to the newly constructed wave consists then of several copies of the potential of the original wave and, to determine the stability properties of the new water wave, one needs to investigate the spectrum of the new energy operator using information from the original operator. In this situation, the fate of embedded eigenvalues under large perturbations (gluing widely separated potentials together is not a small regular perturbation) is the crucial issue that determines the stability of the new waves.

In this paper we focus on the perturbation problem for a self-adjoint differential operator on $L^2(\mathbb{R};\mathbb{R}^n)$.
This is an ODE example of a perturbation problem for embedded eigenvalues, and it serves as a basis for other self-adjoint problems, including more complicated systems of partial differential equations.

Many authors have worked with the problem of finding conditions under which an embedded eigenvalue disappears, see for example \cite{AgHeSki}. On the other hand, we in this case are interested in those potentials for which the embedded eigenvalue exists. Some examples of similar approach to such problems can be found in \cite{AgHeMa,AgHeSki,Astaburuaga,SamHeMart,DeMaSaCylinder,DeMaSaBilaplacian}. More specifically, our goal is to determine the structure of the set of those potentials, i.e. the co-dimension of the manifold of all those perturbations for which the embedded eigenvalue persists. Our method has previously been developed for solving perturbation problems for partial differential operators, as for example in \cite{LaMa,DeMaSaCylinder,DeMaSaBilaplacian}. Other problems concerning persistence of embedded eigenvalues have been studied in \cite{AgHeMa} and \cite{AgHeSki}.


\section{Problem Setup and Main results }

\subsection{Problem setup} 
\vspace{0.5cm}

We are interested in the operator $\mathcal{L}$ in $L^2(\mathbb{R};\mathbb{R}^n)$, such that
\begin{equation}\label{unperturbed}
    \mathcal{L} \coloneqq -\frac{d^2}{dx^2}+ A(x),
\end{equation}
where $A: \mathbb{R}\rightarrow \mathbb{R}^{n \times n}$ is a continuous and symmetric matrix. We also assume that $A(x)$ is asymptotically constant as $|x| \to \infty$ i.e there exists a matrix $A_\infty $ such that $|A(x)-A_\infty| \to 0$ as $|x| \to \infty$. More information on that can be found later on in ~\cref{sec:3}. We assume that the potential $A(x)-A_\infty$ decays algebraically as $|x| \to \infty$ with a sufficiently fast algebraic rate which will be defined in ~\cref{sec:2.2}. 
We also let the eigenvalues of $A_\infty$ be denoted by $a_i$, $i=1,\dots, n$ (counted with multiplicity), and they are numbered in increasing order: $a_1\le a_2\le \dots\le a_n$. This will be beneficial in ~\cref{subsec:3.1}.


 
The eigenvalue equation for the operator $\mathcal{L}$ is 
\begin{equation}\label{E:ev}
    \mathcal{L}\vec{u} = \lambda \vec{u}.
\end{equation}
$\mathcal{L}$ defines a self-adjoint operator, and so its spectrum is a subset of the real line, see \cite[Theorem VI.8, p.194]{ReedSimon}.
We say that $\lambda \in \mathbb{R}$ is an eigenvalue of $\mathcal{L}$ if $(\mathcal L - \lambda I):L^2(\mathbb R;\mathbb R^n)\to L^2(\mathbb R;\mathbb R^n)$ is not injective,
or, equivalently, if $\Ran(\mathcal L-\lambda I)$ is not dense in $\mathcal L(\mathbb R;\mathbb R^n)$ \cite[p. 91]{AkGlaz}. This happens 
if and only if the eigenvalue equation \eqref{E:ev} has a solution belonging to $L^2(\mathbb R;\mathbb R^n)$.

 The continuous spectrum of $\mathcal L$ consists of the $\lambda\in \mathbb R$ such that $\Ran(\mathcal L-\lambda I)$ is not closed \cite[p.88]{AkGlaz}.
An embedded eigenvalue is an 
eigenvalue which also belongs to the continuous spectrum.

 The continuous spectrum of $\mathcal{L}$ is the half line $[\alpha_1, \infty)$ where $\alpha_1$ is the smallest eigenvalue of $A_\infty$. By using rigged Hilbert spaces (see e.g. \cite{Babbitt}), one can show that for the type of operators that we consider, the continuous spectrum  consists of the set of $\lambda\in \mathbb R$ for which there exist bounded solutions of the eigenvalue equation \eqref{E:ev}, that do not belong to $L^2(\mathbb R;\mathbb R^n)$. These solutions are referred to as {\em generalized eigenfunctions}.
 

We assume that we have an embedded eigenvalue $\lambda_0$ for the unperturbed operator $\mathcal{L}$. The following example shows that this assumption can be met.

\begin{example}
Let 
$A(x)$ be a $2 \times 2$ matrix such that
\begin{equation*}
A(x) = 
\begin{pmatrix}
a(x) & 0 \\
0 & -1 \\
\end{pmatrix}.
\end{equation*}
We will show that it is possible to choose $a(x)$ so that $a(x)\to 1$ as $|x|\to \infty$ with an exponential convergence rate so that $0$ is an embedded eigenvalue of $\mathcal L$. 

To find the spectrum of $\mathcal L$ in $\mathcal L(\mathbb R;\mathbb R^n)$, we study the two operators arising from the diagonal elements of $A(x)$, i.e. 
\begin{equation*}
  \begin{aligned}
     \mathcal{L}_1 &= -u''+a(x)u, \\ \mathcal{L}_2 &= -u''-u.
  \end{aligned}
\end{equation*}

$\mathcal L_2$ has continuous spectrum $[-1,\infty)$ with corresponding generalized eigenfunctions $\cos(\sqrt{\lambda+1}x)$ and $\sin(\sqrt{\lambda+1}x)$, ($\lambda>-1$), while 
$\mathcal L_1$ has continuous spectrum $[1,\infty)$ with corresponding generalized eigenfunctions that are asymptotic to $\cos(\sqrt{\lambda-1} x)$ and $\sin(\sqrt{\lambda-1} x)$, $(\lambda>1)$, as $|x|\to \infty$.
The continuous spectrum of $\mathcal L$ is the union of the continuous spectra of $\mathcal L_1$ and $\mathcal L_2$, i.e.
$[-1,\infty)$ with corresponding generalized eigenfunctions $(u_1,u_2)\ne (0,0)$, where $u_i$ is a generalized eigenfunction of $\mathcal L_i$ corresponding to $\lambda$, or $u_i=0$.

We will now construct $a(x)$ so that $0$ is an eigenvalue of $\mathcal L$. Note that the eigenvalue will then be embedded, since $0\in \sigma_c(\mathcal L)=[-1,\infty)$.
This we do by choosing an eigenfunction, for example
 $u_1(x)=1/\cosh x$, and then computing what the corresponding potential $a(x)$ must be.
The eigenvalue equation for $\mathcal L_1$ is $-u_1''+a(x)u=0$, and it is satisfied if
\begin{equation*}
a(x) = 1-\frac{2}{\cosh^2{x}}.
\end{equation*}

Therefore, $\sigma(\mathcal{L}) = [-1,\infty)$ and $\lambda_0=0$ is embedded in the continuous spectrum. The corresponding eigenfunction is $(u_1(x),0)^T$.
\end{example}
In the same manner, it is possible to construct operators of the type \eqref{unperturbed} in $L^2(\mathbb R;\mathbb R^n)$ for an arbitrary $n\ge 2$, which has an embedded eigenvalue. For $n=1$, there are no embedded eigenvalues for the type of potentials that we consider in this paper, since it can be shown that all eigenfunctions decay exponentially (see Lemma \ref{lemma:expdecay}), and if $n=1$, then the solution space of the eigenvalue equation for a particular $\lambda$ is $2$-dimensional. If $\lambda$ is an eigenvalue, there is one solution decaying exponentially as $|x|\to \infty$. The other solutions are linear combinations of this solution and another solution which grows exponentially as $|x|\to\infty$, and so there can be no solutions of the eigenvalue equation which are bounded but not exponentially decaying.

The main theorem of this paper, ~\cref{thm1}, states that there exists a manifold of codimension $2m$ in the space of perturbations so that for $\lambda$ sufficiently close to $\lambda_0$, $\lambda$ is an embedded eigenvalue for the perturbed problem if the perturbation is sufficiently small. We note that by $m$ we denote the number of eigenvalues of $A_\infty$ that are less than $\lambda_0$, and $2m$ is the number of bounded solutions of the eigenvalue equation for $\lambda=\lambda_0$, that are not exponentially decaying, and this number we identify as the multiplicity of the continuous spectrum near $\lambda_0$.
In other words, we analyze the persistence of the eigenvalue $\lambda_0$ when a small perturbation $B$ which decays to 0 with a certain decay rate (see \eqref{eq:decaybanach}) is added to the potential $A$. Therefore, we consider the perturbed operator $\mathcal{L}+B$ which is given by
\begin{equation}\label{eq:lperturbed}
(\mathcal{L}+B)\vec{u}:=-\vec{u''} +[A(x)+B(x)]\vec{u}
\end{equation}
and show that for small $\epsilon > 0$, the set
\begin{equation}\label{eqn:S_eps}
    \mathcal{S}_\epsilon := \{B\in X_\beta\mid \text{there exists } \lambda \in (\lambda_0 - \epsilon, \lambda_0 + \epsilon)\, \text{ such that } \, \lambda \, \text{ is an eigenvalue of } \, \mathcal{L} {+B} \}
\end{equation}
is a smooth manifold of codimension of $2m$ in a neighbourhood of $0\in X_\beta$.
\subsection{Main result}\label{sec:2.2}
\noindent We define the Banach space 
\begin{equation}\label{eq:decaybanach}
X_{\beta} = \{B\in \mathbb{R}^{n \times n}, B \text{ is symmetric} \text{ and } \|B\|_{X_\beta}:=\sup_{x \in \mathbb{R}}(|B(x)|)(1+|x|)^{\beta} < \infty \},
\end{equation}
where $B$ is the matrix multiplication operator that premultiplies the function $\vec{u}$ by the matrix function $B(x)$. It is a bounded operator on $L^2(\mathbb R;\mathbb{R}^n$), and $\beta > 1$. The latter condition is needed as we will later see in ~\cref{lma4}. The space $X_\beta$ specifies the decay rate of $A(x)-A_\infty$ and of the added perturbation $B(x)$.

Let us now make the following assumptions:
\begin{assump}\label{assump1}
We assume that $A(x)-A_\infty\in X_\beta$, for some $\beta>1$.
\end{assump}

\begin{assump}\label{assump2}
$\lambda_0$ is not an eigenvalue of $A_\infty$ i.e. $\lambda_0 \notin \{\alpha_0, \dots,\alpha_n \}$ 
\end{assump}
%
%
%
We define $m$ to be the number of eigenvalues of $A_\infty$ that are less than $\lambda_0$.
We are now ready to state our main result:

\begin{theorem} \label{thm1}
Let $\mathcal{L}$ be the operator defined in ~\eqref{unperturbed} and let $\lambda_0$ be an eigenvalue of $\mathcal{L}$. Suppose that Assumptions 1 and 2 hold. Let $A_\infty$ and $m$ be as above, and let $S_\epsilon$ be as in \eqref{eqn:S_eps}.
Then there exists an $\epsilon>0$ and a neighbourhood $\mathcal{N}$ of $0\in X_\beta$,\,  such that\, $\mathcal{S_{\epsilon} \cap \mathcal{N}}$ is a manifold of codimension $2m$ in $X_\beta$.
\end{theorem}

In other words, ~\cref{thm1} provides the desired characterization of the set of perturbations which do not remove the embedded eigenvalue $\lambda$ of the perturbed operator $\mathcal{L}+B$. 

The methods used in this paper are based in the ones in \cite{DeMaSaCylinder} and \cite{DeMaSaBilaplacian} that are originally formulated for the PDE perturbation problems. To solve this perturbation problem, see ~\eqref{eq:lperturbed}, we first write the eigenvalue problem as a system of first order ODEs, ~\eqref{fullsystem}. It will be shown that this system, after a necessary shift $\eta$ (see ~\cref{lma1}), has  exponential dichotomies at $\mathbb R_+$ and $\mathbb R_-$, i.e. for each $x \in \mathbb{R}$
it has a stable subspace that consists of initial values (at $x$) corresponding to the exponentially decaying solutions at $x\to\infty$, while solutions not in this subspace instead grow exponentially as
 $x\to\infty$. Similarly, there exists an unstable subspace of initial values (at $x$) corresponding to the exponentially decaying solutions as $x\to-\infty$, while all the other solutions instead grow exponentially as $x\to -\infty$, see ~\cref{lma2} and ~\cref{lemma:expdecay}. The main tool for deriving this result is ~\cref{theorem:roughness}. The eigenvalue problem can then be translated into the problem of determining whether these stable and unstable subspaces intersect. That is done by using Lyapunov-Schmidt reduction. Indeed we show that for small perturbations, there are $2m$ conditions that need to be satisfied in order for the exponentially decaying solutions at $+\infty$ connect to the ones at $-\infty$, see ~\cref{proof:main}, Proof of Theorem 2.1. Finally, we use the implicit function theorem together with the fact that the unperturbed system solves the equation, to prove that these can be solved in a neighbourhood of $\lambda_0$.

\section{The ODE formulation}\label{sec:3}
\noindent
In this section we study the system as $|x| \to 
\infty$ and we introduce the notion of exponential dichotomies.
We will see in ~\cref{lemma:expdecay} that any eigenfunction decays exponentially. Initial values of asymptotically decaying solutions of non-autonomous linear systems such as the one we are interested in, ~\eqref{fullsystem}, can be found as intersections of stable and unstable subspaces. In order to investigate those stable and unstable subspaces, the system at infinity and the concept of exponential dichotomies are introduced.

  \subsection{The system at infinity}\label{subsec:3.1}

The eigenvalue equation for the perturbed operator $\mathcal{L} + B$ corresponding to the eigenvalue $\lambda$ is

\begin{equation}\label{eq:eigenvalueproblemforLB}
-\vec{u}''+(A(x)+B(x))\vec{u}=\lambda \vec{u}.
\end{equation}
It will be beneficial to write \eqref{eq:eigenvalueproblemforLB} as system of first order ODEs. To do that, we set $\vec{u}=\vec{u}_1$ and $\vec{u}'=\vec{u}_2$ and obtain
\begin{equation}\label{fullsystem}
U'=M(x;\lambda,B)U
\end{equation}
where, $U=(\vec{u}_1,\vec{u}_2)^{T} \in \mathbb{R}^{2n}$ and
$$M(x;\lambda,B)=\begin{bmatrix}
0 & I\\
A(x)+B(x)-\lambda I & 0
\end{bmatrix}.$$

As $|x| \to \infty$, $A(x) \to A_{\infty}$ and $B(x) \to 0$. By replacing $A(x)$ and $B(x)$ by these limits, we obtain the system at infinity which is given by 
\begin{equation} \label{systeminfinity}
U'=M_{\infty}(\lambda) U,
\end{equation}
where $M_{\infty}(\lambda) =\begin{bmatrix}
0 & I\\
A_{\infty}-\lambda I & 0
\end{bmatrix}$.
Its solutions capture the asymptotic behaviour of the solutions of our unperturbed system, i.e. \eqref{fullsystem}, which we will see later in Section 4.  

Now some comments on the set of eigenvalues of $M_{\infty}(\lambda)$.

\noindent As $A_{\infty}$ is real and symmetric, it is diagonalizable by an orthogonal matrix, and so  $A_{\infty}=Q^T D Q$ where $D$ is diagonal and $Q^T Q= I$.
We assume that $\lambda$ belongs to an interval which is small enough so that the sign of $\lambda-a_i$ doesn't change for $\lambda$ in this interval.
\noindent From ~\cref{assump2}, it follows that the eigenvalues of $M_{\infty}(\lambda)$ are the pairs 
\begin{equation*}
     \mu =\begin{cases}
        \pm \sqrt{a_i-\lambda} &\text{for }i\ge m+1, \\
        \pm i\sqrt{\lambda-a_i} &\text{for }i\le m.
     \end{cases}
\end{equation*}
where $i=1,\dots,n$. In particular, there are $2m$ imaginary eigenvalues and $2(n-m)$ real ones.
We shall denote by $\mu_{\min} := \sqrt{a_{m+1}-\lambda_0}$, the smallest positive  eigenvalue of $M_\infty(\lambda_0)$.

Let $X^u$ and $X^s$ be the span of eigenfunctions corresponding to positive and negative eigenvalues of $M_{\infty}(\lambda_0)$ respectively. Let us also denote by $X^c$ the span of eigenfunctions corresponding to the purely imaginary eigenvalues. Let $P^u, P^s$ and $P^c$ be the spectral projections onto $X^u$, $X^s$ and $X^c$ respectively.

\subsection{Exponential dichotomies}
Exponential dichotomies is the main tool for proving the main result, ~\cref{thm1}. In this section we introduce this concept and show that after a slight modification, our perturbed and unperturbed systems possess  exponential dichotomies. This is done by first proving that the system at infinity, also after a corresponding modification, possesses an exponential dichotomy, and then using a perturbation result, the Roughness theorem, Theorem \ref{theorem:roughness} to get the desired result.

\begin{definition}\label{D:expdich}
An ODE system $U'=C(x)U$ is said to possess an exponential dichotomy on $J$, where $J$ is an unbounded interval of $\mathbb R$, if there exist constants $K>0, \kappa^s <0< \kappa^u$ and a family of projections $P(x_0)$ 
such that:
\begin{itemize}
\item For any $x\in\mathbb{R}$ and $U \in \mathbb{R}^{2n}$, there exists a unique solution $\Phi^s(x,x_0)U$ of the system defined for $x \geq x_0$, $x,x_0 \in J$ such that
\begin{equation*}
 \Phi^s(x_0,x_0)U=P(x_0)U \quad \text{and} \quad \|\Phi^s(x,x_0)U\| \leq Ke^{\kappa^s(x-x_0)}\|U\|.
\end{equation*}
\item  For any $x\in\mathbb{R}$ and $U$ there exists unique solution $\Phi^u(x,x_0)U$ of the system defined for $x \leq x_0$, $x,x_0 \in J$ such that
\begin{equation*}
\Phi^u(x_0,x_0)U=(I-P(x_0))U \quad \text{and} \quad \|\Phi^u(x,x_0)U\| \leq Ke^{\kappa^u(x-x_0)}\|U\|.
\end{equation*}
\item The solutions $\Phi^s(x,x_0)U$ and $\Phi^u(x,x_0)U$ satisfy
\begin{gather*}
\Phi^s(x,x_0)U \in \Ran P(x) \quad \text{for all} \quad x \geq x_0, \quad x,x_0 \in J\\
\Phi^u(x,x_0)U \in \ker P(x) \quad \text{for all} \quad x \leq x_0, \quad x,x_0 \in J.
\end{gather*}
\end{itemize}
\end{definition}
In order to use \cref{D:expdich} for our system, we need to introduce a shift, $\eta$, which makes the eigenvalues of $M_\infty(\lambda_0)$ move to the right or left, so that they avoid the imaginary axis.
\begin{lemma}\label{lma1}
Let $J=\mathbb R, \mathbb R_+$ or $\mathbb R_-$. Suppose that $\eta \in (0, \mu_{\min})$. Then the systems {$U'=(M_\infty(\lambda_0)+\eta I)U$ and $U'=(M_\infty(\lambda_0)-\eta I)U$} each possess exponential dichotomies on $J$, with $\kappa^s = - \mu_{min}+ \eta$, $\kappa^u = \eta$ and $\kappa^s = -\eta$, $\kappa^u = \mu_{min}-\eta$ respectively. 
\end{lemma}
\begin{proof}
   For the system $U'=(M_\infty(\lambda_0) + \eta I) U$,
   let
   \begin{equation*}
     \left\{
      \begin{aligned}
       \Psi_\infty^s(x,x_0) &= e^{(M_\infty(\lambda_0) + \eta I)P^s}P^s, \\
       \Psi_\infty^{cu}(x,x_0) &= e^{(M_\infty(\lambda_0) + \eta I)(P^u + P^c)}(P^u + P^c).
      \end{aligned}
      \right.
   \end{equation*}
   Then the requirements of Definition \ref{D:expdich} are satisfied with $P=P^s$ and $\kappa^s=-\mu_{\min} + \eta<0$, $\kappa^u= \eta>0$.
   
   For the system $U' = (M_\infty(\lambda_0)-\eta I) U$,  let
   \begin{equation*}
     \left\{
      \begin{aligned}
       \Psi^{cs}_{\infty}(x,x_0) &= e^{(M_\infty(\lambda_0) - \eta I)(P^s+P^c)}(P^s+P^c), \\
       \Psi^u_{\infty}(x,x_0) &= e^{(M_\infty(\lambda_0) - \eta I)P^u}P^u.
      \end{aligned}
      \right.
   \end{equation*}
   Then the requirements of Definition \ref{D:expdich} are satisfied with $P= P^s + P^c$ and $\kappa^s = -\eta<0$, $\kappa^u = \mu_{\min}-\eta>0$.
\end{proof}

Next, we study the full system ~\eqref{fullsystem}, which can be expressed as
\begin{equation}\label{fullsystemalternative}
U' = (M_{\infty}(\lambda_0) + L(x;\lambda,B))U
\end{equation}
\noindent where
\begin{equation}\label{matrixL}
L(x;\lambda,B) = \begin{bmatrix}
0 & 0\\
A(x)- A_{\infty}+(\lambda_0 - \lambda)I  + B(x)& 0
\end{bmatrix}.
\end{equation}
One of the most important properties that exponential dichotomies possess is their roughness. By that, we mean that they persist even if we add a perturbation in the coefficient matrix, which is small for all large $x$. This property is proved in the following lemma. For more information on the topic we refer to \cite{Coppel78}.

\begin{lemma}{Roughness Theorem}\label{theorem:roughness} \label{T:roughness}
\begin{enumerate}[(i)]
\item If $U'=C(x) U$ possesses an exponential dichotomy on $\mathbb R_+$ with rates $\kappa^s<0<\kappa^u$ and constant $K>0$ as in Definition \ref{D:expdich}, and if for some $R>0$, $|D(x)| < \delta$ for all $x\ge R$, where $\delta\in (0,\min(-\kappa^s,\kappa^u)/(2K)$, then the perturbed system $U'=(C(x)+D(x)) U$ also possesses an exponential dichotomy on $\mathbb R_+$ with rates $\tilde \kappa^s=\kappa^s+2K\delta<0$, $\tilde \kappa^u = \kappa^u - 2K\delta>0$ and some constant $\widetilde K> 0$.
\item If $U'=C(x) U$ possesses an exponential dichotomy on $\mathbb R_-$ with rates $\kappa^s<0<\kappa^u$ and constant $K>0$ as in Definition \ref{D:expdich}, and if for some $R>0$, $|D(x)|<\delta$ for all $x<-R$, where $\delta\in (0,\min(-\kappa^s,\kappa^u)/(2K)$, then  the perturbed system $U'=(C(x)+D(x))U$ also possesses an exponential dichotomy on $\mathbb R_-$, with rates $\tilde \kappa^s=\kappa^u+2K\delta$, $\tilde \kappa^u=\kappa^u-2K\delta$ and some  constant $\widetilde K$.
\end{enumerate}
\end{lemma}

\begin{proof}
We shall prove the first statement, since the second one can be proven in the same way.

In \cite[p.34]{Coppel78} we see that if an unperturbed system $U'=C(x)U$ has an exponential dichotomy on $\mathbb R_+$  then if $\delta\in (0,\min(-\kappa^s,\kappa^u)/(2K)$ such that if $\sup_{x\ge 0} |D(x)|<\delta$, then 
$U'=(C(x)+D(x))U$ has an exponential dichotomy on $\mathbb R_+$ with the required rates. 

In our case, $|D(x)|$ is not small for all $x\ge 0$, and so the proposition cannot be directly applied. Note however that since $|D(x)|<\delta$ for all $x>R$,  $\sup_{x\ge R}|D(x)|<\delta$ (with $\delta$ as above). Using the above result (with a transformed $x$ variable, $\widetilde x= x-R$), it follows that the perturbed system has an exponential dichotomy for $x\ge R$, and we denote the corresponding operators with $\Phi^s(x,x_0)$ (defined for $x\ge x_0\ge R$) and $\Phi^u(x,x_0)$ (defined for $x_0\ge x\ge R$). 

These dichotomies can be extended to be a dichotomy on the whole of $\mathbb R_+$ as follows:
The evolution operator $\Phi(x,x_0)$ is defined for all $x$, $x_0\in \mathbb R$, and it is defined as the unique solution of
\begin{equation*}
    \left\{
    \begin{aligned}
       \Phi'(x,x_0) &= (C(x)+D(x)) \Phi(x,x_0), \\
       \Phi(x_0,x_0) &= I.
    \end{aligned}
    \right.
\end{equation*}

For $0\le x_0<R$, $x\ge x_0$, we {\em define}
\begin{equation*}
    \Phi^s(x,x_0) = \Phi(x,R)P(R)\Phi(R,x_0),
\end{equation*}
where $P(x)$ is the projection related to the exponential dichotomy for the perturbed system $U'=(C(x)+D(x))U$, and which exists for $x\ge R$ by the above result.
Likewise, for $0\le x<R$, $x_0>x$, we {\em define}
\begin{equation*}
    \Phi^u(x,x_0) = \Phi(x,R)(I-P(R))\Phi(R,x_0).
\end{equation*}
Let $P(x)=\Phi^s(x_0,x_0)$ also for $0\le x_0\le R$.
It is not difficult to check that 
$\Phi^s$, $\Phi^u$ are solutions and that 
$\Phi^s$, $\Phi^u$ have the required properties of Definition \ref{D:expdich}. 
\end{proof}
\begin{lemma}\label{lma2}
 Let $\eta\in (0,\mu_{\min})$ and let $\epsilon>0$ be arbitrary. Then there exists $\delta>0$ such that if $|\lambda-\lambda_0| + \sup_{x\in\mathbb R}\|B(x)\|<\delta$, then he systems
\begin{equation} \label{systemswithL}
V'_{\pm} = (M_{\infty}(\lambda_0) \pm \eta I + L(x;\lambda,B))V_{\pm}
\end{equation}
\noindent possess exponential dichotomies  on $\mathbb{R}_+$ and $\mathbb{R}_-$, respectively.
\begin{enumerate}[(i)]
\item For the case of $V_+$, the system has an exponential dichotomy  on $\mathbb R_+$ with rates $\kappa^s = -\mu_{min}+\eta+\epsilon$, $\kappa^u = \eta-\epsilon$.
\item For the case of $V_-$ on $\mathbb R_-$, the system has an exponential dichotomy with rates: $\kappa^s =  -\eta+\epsilon$ and $\kappa^u= \mu_{min}-\eta-\epsilon$.
\end{enumerate}
\end{lemma}

\noindent For $V_+$ on $\mathbb R_+$, we denote the projections by $P^s(\cdot;\lambda,B)$, and we let $P^{cu}(\cdot;\lambda,B):=I-P^s(\cdot;\lambda,B)$.
We denote the corresponding evolution operators on $\mathbb R_+$ by $\Psi^s(x,x_0;\lambda,B)$ and $\Psi^{cu}(x,x_0;\lambda,B)$.

\noindent For $V_-$ on $\mathbb R_-$, we denote the projections by $P^{cs}(\cdot;\lambda,B)$ and we let $P^u(\cdot,\lambda,B):=I-P^{cs}(\cdot;\lambda,B)$. We denote the corresponding evolution operators on $\mathbb R_-$ by
$\Psi^{cs}(x,x_0;\lambda,B)$ and $\Psi^u(x,x_0;\lambda,B)$.
\begin{proof}
The result follows directly from Lemma \ref{lma1} together with Lemma \ref{T:roughness}.
\end{proof}
\begin{lemma} \label{smoothevolutionop}
The projections $P^s(\cdot;\lambda,B)$, $P^{cu}(\cdot;\lambda,B)$, $P^{cs}(\cdot;\lambda,B)$, $P^u(\cdot;\lambda,B)$  and the corresponding evolution operators $\Psi^s(\cdot,\cdot,\lambda,B)$, $\Psi^{cu}(\cdot,\cdot,\lambda,B)$, $\Psi^{cs}(\cdot,\cdot,\lambda,B)$ and $\Psi^u(\cdot,\cdot,\lambda,B)$ depend smoothly on the parameters $\lambda$ and $B$ in a neighbourhood of $(\lambda,B) \in \mathbb{R} \times X_\beta$.
\end{lemma}

\begin{proof}
 In \cite[p.30]{Coppel78} it has been shown that exponential dichotomies can be expressed as fixed points of a specific affine map $\mathcal T:L^{\infty}(\mathbb R^{2n})\to L^{\infty}(\mathbb R^{2n})$, where $\mathcal T$ is proved to be a contraction if the matrix $D(x)$ given in Lemma \ref{theorem:roughness} is small enough in the $L^\infty$ norm. Hence $I-T$ is invertible in a neighbourhood of $(\lambda_0,0)$.  But the operator depends smoothly on $(\lambda,B)$, since our matrix $L(x;\lambda,B)$ is smooth. This allows us to apply the implicit function theorem and obtain that the solutions to this equation depend smoothly on the parameters $\lambda$ and $B$. 
\end{proof}
Now using ~\cref{lma2} we can define the following evolution operators for the system (6) on $\mathbb{R}_+$ and $\mathbb{R}_-$:

On $\mathbb R_+$, we will use the operators $\Phi^s$ and $\Phi^{cu}$ defined by

\begin{equation*}
   \begin{aligned}
\Phi^s(x,x_0;\lambda, B) &= e^{-\eta(x-x_0)}\Psi^s(x,x_0;\lambda,B), \\
\Phi^{cu}(x,x_0;\lambda,B) &= e^{-\eta(x-x_0)}\Psi^{cu}(x,x_0;\lambda,B).
   \end{aligned}
\end{equation*}

On $\mathbb R_-$, we will use the operators $\Phi^{cs}$ and $\Phi^u$ defined by

\begin{equation*}
\begin{aligned}
   \Phi^{cs}(x,x_0;\lambda,B) &= e^{\eta(x-x_0)}\Psi^{cs}(x,x_0;\lambda,B), \\
    \Phi^u(x,x_0;\lambda,B) &= e^{\eta(x-x_0)}\Psi^u(x,x_0;\lambda,B).
   \end{aligned}
\end{equation*}

\section{Exponential decay of eigenfunctions}
\noindent In this section we show that every eigenfunction is exponentially decaying as $|x| \to \infty$.

We first give an expression for the bounded solutions of the eigenvalue equation ~\eqref{fullsystem}. This formula is then used to prove that eigenfunctions decay exponentially.
\begin{lemma}\label{lma4}
Let $U$ be a solution of ~\eqref{fullsystem}. 
\begin{enumerate}[(i)]
   \item{If $U$ is bounded on $\mathbb{R}_+$, then for every $R \geq 0$, there exists a $U_0^s \in X^s$, and a $U_0^c \in X^c$ such that for all $x \geq R$
\begin{equation*}
  \begin{aligned}
    U(x)=e^{M_\infty P^s(x-R)}U_0^s&+e^{M_\infty P^c x}U_0^c+\int_R^x e^{M_\infty P^s(x-\xi)}P^s L(\xi;\lambda_0,B)U(\xi)\, d\xi \\
    &-\int_x^\infty e^{M_\infty P^{cu}(x-\xi)}P^{cu} L(\xi;\lambda_0,B)U(\xi)\, d\xi,
  \end{aligned}
\end{equation*}
where $X^c=X^c(\lambda)$, $X^s=X^s(\lambda)$ are the closures of the span of eigenvectors of $M_\infty=M_\infty(\lambda)$ corresponding to the purely imaginary and negative eigenvalues of $M_\infty(\lambda)$, respectively.}
\item If $U$ is bounded on $\mathbb R_-$, then for every $R \geq 0$, there exists a $V_0^s \in X^s$, and a $V_0^c \in X^c$ such that for all $x \leq -R$
\begin{equation*}
  \begin{aligned}
    U(x)=e^{M_\infty P^u(x+R)}V_0^u&+e^{M_\infty P^c x}V_0^c-\int_x^{-R} e^{M_\infty P^u(x-\xi)}P^u L(\xi;\lambda_0,B)U(\xi)\, d\xi \\
    &+\int_{-\infty}^x e^{M_\infty P^{cs}(x-\xi)}P^{cs} L(\xi;\lambda_0,B)U(\xi)\, d\xi.
  \end{aligned}
\end{equation*}
\end{enumerate}
\end{lemma}

\begin{proof}
We only prove (i), since (ii) can be proved in a similar manner.

We begin with the full ODE ~\eqref{fullsystemalternative} and project using $P^s$, $P^c$ and $P^u$ (where we have suppressed $\lambda$ for convenience):
\begin{equation*}
     \left\{
      \begin{aligned}
         P^s U'(x) = M_{\infty}P^sU(x) + P^sL(x;\lambda_0,B)U(x), \\
         P^c U'(x) = M_{\infty}P^cU(x) + P^cL(x;\lambda_0,B)U(x), \\
         P^u U'(x) = M_{\infty}P^uU(x) + P^uL(x;\lambda_0,B)U(x),
      \end{aligned}
      \right.
   \end{equation*}
   where we have used that the projections $P^s$, $P^c$ and $P^u$ commute with $M_\infty$.

We view the above equations as inhomogeneous versions of the system at infinity. Hence the variation of constants formula can be used to express the solutions as
\begin{equation}\label{solutionsVariationofConstants}
     \left\{
      \begin{aligned}
         P^s U(x) = e^{M_\infty P^s(x-x_0)}  P^s U(x_0)+\int_{x_0}^{x}e^{M_\infty P^s(x-\xi)}P^s L(\xi;\lambda_0,B)U(\xi) \,d\xi \\
         P^c U(x) = e^{M_\infty P^c(x-x_0)} P^c U(x_0)+\int_{x_0}^{x}e^{M_\infty P^c(x-\xi)}P^c L(\xi;\lambda_0,B)U(\xi) \,d\xi \\
         P^u U(x) = e^{M_\infty P^u(x-x_0)} P^u U(x_0)+\int_{x_0}^{x}e^{M_\infty P^u(x-\xi)}P^u L(\xi;\lambda_0,B)U(\xi) \,d\xi.
      \end{aligned}
      \right.
   \end{equation}
\noindent Note that $P^s U(x)$, $P^c U(x)$ and $P^u U(x)$ are bounded since $\norm{U(x)}$ is bounded as $x \to \infty$.

We first look at $P^u U(x)$ and let $x_0 \to \infty$ in the last equation of ~\eqref{solutionsVariationofConstants}. Since $P^u U(x_0)$ is bounded, the first term converges to $0$ and  it follows that
\begin{equation*}
    P^u U(x) = -\int_{x}^{\infty}e^{M_\infty(\lambda)P^u(x-\xi)}P^u L(\xi;\lambda_0,B)U(\xi) \,d\xi.
\end{equation*}

\noindent Next, we study the equation for $P^c U(x)$. The integral in the second equation of ~\eqref{solutionsVariationofConstants} converges as $ x_0 \to \infty$, which can be shown in the following way: Since $\norm{e^{M_\infty(\lambda)P^c(x_0)}}$ is bounded for $x_0 \in \mathbb{R}$ and since
$L(\xi,\lambda_0,B)U(\xi)=(A(\xi)-A_\infty+B(\xi))u_1$, which implies that $\|L(\xi,\lambda_0,B)U(\xi)\|\le \|A-A_\infty+B\|_{X_\beta}\|U(\xi)\|(1+\xi)^{-\beta}$,
we have
\begin{equation}\label{eqEstimate}
\begin{aligned}
  \int_{x}^{\infty} \|L(\xi;\lambda_0,B)U(\xi)\| \,d\xi \leq (\|A-A_\infty\|_{X_\beta}+\|B\|_{X_\beta})\sup_{\xi\ge x}\|U(\xi)\| \int_{x}^{\infty} {(1+\xi)^{-\beta}} \,d\xi \\
    \le \frac{1}{\beta-1}(\|A-A_\infty\|_{X_\beta}+\|B\|_{X_\beta}) \|U\| \frac{1}{(1+x)^{\beta-1}}.
\end{aligned}
\end{equation}
\noindent For the other term of the same equation of ~\eqref{solutionsVariationofConstants}, we observe that the limit 
\begin{equation*}
    \lim_{x_
0 \to \infty} e^{-M_{\infty}(\lambda)P^c x_0}U(x_0) =: U^c_0
\end{equation*}
 exists since $P^c U(x)$ does not depend on $x_0$ and the integral ~\eqref{eqEstimate} converges.
 
\noindent Thus,
\begin{equation*}
    P^c U(x) = e^{M_\infty P^c x}U^c_0 - \int_{x}^{\infty} e^{M_\infty(\lambda)P^c(x-\xi)}P^c L(\xi;\lambda_0,B)U(\xi) \,d\xi.
\end{equation*}
\noindent For $P^s U(x)$ we choose $x_0=R \geq 0$ so that from ~\eqref{solutionsVariationofConstants} for $x \geq R$ we have
\begin{equation*}
\begin{aligned}
    U(x) = e^{M_\infty P^s(x-R)}U^s(R) + e^{M_\infty P^c x}U^c_0 + \int_{R}^{x} e^{M_\infty P^s(x-\xi)}P^s L(\xi;\lambda_0,B)U(\xi) \, d\xi \\
    - \int_{x}^{\infty} e^{M_\infty P^{cu}(x-\xi)}P^{cu}L(\xi;\lambda_0,B)U(\xi) \, d\xi.
\end{aligned}
\end{equation*}

\noindent Lastly, we write $U^s(R)=U^s_0$ to obtain the desired formula.
\end{proof}

Next, by using ~\cref{lma4} and a contraction mapping argument we shall prove exponential decay of any eigenfunction. 

\begin{lemma}\label{lemma:expdecay}
Let $\lambda$ be an eigenvalue of the perturbed operator ~\eqref{eq:lperturbed} with $\lambda \notin \{\alpha_1,\dots, \alpha_n \}$ and corresponding eigenfunction $\vec{u} \in L^2(\mathbb{R};\mathbb{R}^n)$. Let also $A(x)-A_\infty \in X_\beta$ and $\hat{\kappa} \in (0,\mu_{\min}(\lambda))$. Denote by $U$ the solution of the system $~\eqref{fullsystem}$. Then, there exists a positive constant $K$ such that
\begin{equation*}
\|U(x)\| \leq K e^{-\hat{\kappa}|x|} \quad \text{for all} \;\; x\in \mathbb{R}.
\end{equation*}
\end{lemma}

\begin{proof}
We will concentrate on the proof for $x \rightarrow +\infty$. The argument for $x \to -\infty$ is similar. Let us first estimate the integrals in ~\cref{lma4} (i). Let
\begin{equation*}
    I_1 = \int_R^x e^{M_\infty P^s(x-\xi)}P^s L(\xi;\lambda_0,B)U(\xi)\; d\xi
\end{equation*}
and
\begin{equation*}
    I_2 = \int_x^\infty e^{M_\infty P^{cu}(x-\xi)}P^{cu} L(\xi;\lambda_0,B)U(\xi)\; d\xi.
\end{equation*}

\noindent The aim is to prove that $I_1 \rightarrow 0$ and $I_2 \rightarrow 0$ as $x\rightarrow \infty$.
Let us start with the integral $I_2$.

\noindent By ~\eqref{eqEstimate} and since $\norm{e^{M_\infty P^{cu}x}P^{cu}} \leq 1$ we get
\begin{equation*}
    \norm{I_2} \leq \frac{1}{\beta-1}(\|A-A_\infty\|_{X_\beta}+\|B\|_{X_\beta})  \frac{1}{(1+x)^{\beta-1}},
\end{equation*}
and so $I_2\to 0$ as $x\to \infty$.
Next we shall estimate the integral $I_1$.
For that, we will use that for any $\alpha \in \mathbb{R}$ the following holds:
\begin{equation*}
    \lim_{r \to \infty} \int_1^r \left(\frac{r}{s} \right)^{\alpha} e^{-(r-s)} \; ds = 1 ,
\end{equation*}
which can be verified for example by the $\infty/\infty$ form of L'H\^opital's rule. 
This implies that there exists a constant $C>0$ such that for every $r \geq 1$
\begin{equation*}
    \int_1^r \left( \frac{r}{s} \right)^\alpha e^{-(r-s)} \; ds \leq C.
\end{equation*}

Set $\kappa=\mu_{\min(\lambda)}(= \sqrt{a_{m+1}-\lambda})$ and $\kappa(1+\xi)=\tau$. Then by the above, 
\begin{equation*}
    \begin{aligned}
    \norm{I_1} \leq \frac{\norm{A-A_{\infty}}_{X_\beta} +\norm{B}_{X_\beta}}{(1+x)^\beta}\norm{U} \int_R^x e^{-\kappa(x-\xi)} \left(\frac{1+x}{1+\xi} \right)^\beta \; d\xi \\
    = \frac{(\norm{A-A_{\infty}}_{X_\beta} +\norm{B}_{X_\beta}) \norm{U}}{(1+x)^\beta} \int_{\kappa(1+R)}^{\kappa(1+x)} \left( \frac{\kappa(1+x)}{\tau} \right)^\beta e^{-(\kappa(1+x)-\tau)} \; d\tau \\
    \leq C\frac{(\norm{A-A_{\infty}}_{X_\beta} +\norm{B}_{X_\beta} )\norm{U}}{(1+x)^\beta}.
    \end{aligned}
\end{equation*}
\noindent This completes the proof that $I_1$ and $I_2$ converge to zero as $x \to \infty$. Then by ~\cref{lma4} (i), $U(x) \to  U^c_0$ as $x \to \infty$ and since $U(x) \to 0$ as $x \to \infty$ (which holds since $\vec u$ is an eigenfunction), it follows that $U^c_0 = 0$.

\noindent Thus, for $x \geq R \geq 1/\kappa$, the integral equation in ~\cref{lma4}(i) becomes

\begin{equation}\label{solutionlemma4.1}
  \begin{aligned}
      U(x)=e^{M_\infty P^s(x-R)}U^s_0 &+\int_R^x e^{M_\infty P^s(x-\xi)}P^s L(\xi;\lambda_0,B)U(\xi)\; d\xi \\
      &-\int_x^\infty e^{M_\infty P^{cu}(x-\xi)}P^{cu}L(\xi;\lambda_0,B)U(\xi)\; d\xi.
  \end{aligned}
\end{equation}

Next, we define the spaces
\begin{equation*}
Y_\eta=\Set{\widehat{U}}{\|\widehat{U}\|_{Y_\eta}:={\sup_{x \geq R}e^{\eta x}\|\widehat{U}(x)\|<\infty}}, \quad \eta \geq 0.
\end{equation*}
Then for $\eta \in [0,\kappa)$ let $T: Y_\eta \to Y_\eta$ be given by
\begin{equation*}
  \begin{aligned}
     T(U)(x)=e^{M_\infty P^s(x-R)}U^s_0 &+\int_R^x e^{M_\infty P^s(x-\xi)}P^s L(\xi;\lambda_0,B)U(\xi)\; d\xi \\
     &-\int_x^\infty e^{M_\infty P^{cu}(x-\xi)}P^{cu}L(\xi;\lambda_0,B)U(\xi)\; d\xi.  
  \end{aligned}
\end{equation*}
We shall prove that $T$ is a contraction on $Y_\eta$.

\noindent Let $\widehat{U}_1, \widehat{U}_2 \in Y_\eta$. It follows that
\begin{equation*}
\begin{aligned}
\|T\widehat{U}_1-T\widehat{U}_2\|_{Y_\eta} \leq \|\widehat{U}_1-\widehat{U}_2\|_{Y_\eta} \left(\sup_{x \geq R} \int_R^x \norm{e^{M_\infty P^s(x-\xi)}P^s} \norm{L(\xi;\lambda_0,B)}e^{\eta(x-\xi)} \; d\xi \right. \\
\left. + \sup_{x\geq R} \int_x^{\infty} \norm{e^{M_\infty P^{cu}(x-\xi)}P^{cu}} \norm{L(\xi;\lambda_0,B)} e^{\eta(x-\xi)} \; d\xi\right) \\
\leq K \norm{L}\norm{\widehat{U}_1-\widehat{U}_2} \left(\sup_{x\geq R} \frac{1}{(1+x)^\beta} \int_R^x e^{-(\kappa-\eta)(x-\xi)} \left(\frac{1+x}{1+\xi}\right)^\beta \; d\xi \right. \\
\left. + \sup_{x\geq R} \int_x^\infty \frac{1}{(1+\xi)^\beta} \; d\xi \right) \\
\leq K \norm{L}\norm{\widehat{U}_1-\widehat{U}_2} \left(\frac{C}{(\kappa-\eta)(1+R)^\beta}+\frac{1}{(\beta-1)(1+R)^{\beta-1}} \right),
\end{aligned}
\end{equation*}
where $K=\max(\kappa, 1/\kappa)$.
Therefore, by choosing $R$ large enough we have that $T$ is a contraction and there exists a unique fixed point in $Y_\eta$ for any $\eta \in [0,\kappa)$.

\noindent Since $U \in Y_0$ solves ~\eqref{solutionlemma4.1} we conclude that it is a fixed point for $T$ when $\eta=0$. By uniqueness of the fixed point in $Y_0$, those fixed points must be the same. Thus, $U \in Y_\eta$ for any $\eta \in [0,\kappa)$. This leads us to the desired result, that is, $U$ decays exponentially.

\end{proof}

\section{Lyapunov--Schmidt reduction}
\noindent
In this section we prove the main result. For this the Lyapunov-Schmidt reduction method will be used.

Let $\vec{u_*}$ be the eigenfunction associated with the unperturbed problem $\mathcal{L}\vec{u_*}=\lambda_0 \vec{u_*}$. Let us also assume that $\vec{u_*}$ is normalized so that $\int_{-\infty}^{+\infty} \norm{\vec{u_*}(x)}^2=1$. We denote by $U_* = (\vec{u_*},\vec{u_*'})^T$
the particular solution corresponding to the eigenfunction $\vec{u_*}$ of problem ~\eqref{fullsystem} with $B=0$ and $\lambda=\lambda_0$. 

Now let us define the stable and unstable subspaces $E^s_+$ and $E^u_-$ respectively. Roughly speaking, these subspaces consist of the initial conditions for which the solutions of the unperturbed system decay exponentially in forward and backward time (we think of the $x$ variable as time). Here we use ~\cref{lemma:expdecay}. We also note that $E^s_+ \cap E^u_- = \spn{\{U_*(0)\}}$ because $\lambda_0$ is an embedded eigenvalue. More specifically, we define
\begin{equation}\label{StableUnstableSubspaces}
\begin{aligned}
E^s_+ := \{ U \in \mathbb R^{2n}\mid P^s (R;\lambda_0,0)U = U \},\\
E^u_- := \{ U \in \mathbb R^{2n}\mid P^u (-R;\lambda_0,0)U = U \}.
\end{aligned}
\end{equation}
\noindent To find embedded eigenvalues we shall define a mapping $\iota: E^s_+ \times E^u_- \times \mathbb{R} \times X_\beta \to \mathbb{R}^{2n}$ such that
\begin{equation}\label{iotadef}
\iota(U^s_0,U^u_0;\lambda,B) = \Phi(0,R;\lambda,B)P^s(R;\lambda,B) U^s_0 - \Phi(0,-R;\lambda,B)P^u(-R;\lambda,B) U^u_0.
\end{equation}

\begin{lemma} \label{lma6}
Let $\lambda \notin \{ \alpha_1,\dots,\alpha_n\}$ and let $\delta>0$ be as in Lemma \ref{lma2}. Then $\lambda$ is an eigenvalue of $\mathcal{L}_B$ if and only if there exist $U^s_0 \in E^s_+$ and $U^u_0 \in E^u_-$ with $(U^s_0,U^u_0) \neq 0$ such that
\begin{equation}\label{iotafunction}
\iota(U^s_0,U^u_0;\lambda,B) = 0.
\end{equation}
\end{lemma}
\begin{proof}
Let ~\eqref{iotafunction} hold. Then by definition
\begin{equation*}
\Phi(0,R;\lambda,B)P^s(R;\lambda,B)U^s_0 = \Phi(0,-R;\lambda,B)P^u(-R;\lambda,B)U^u_0,
\end{equation*}
which implies that the solution of ~\eqref{fullsystem}  with initial value
\begin{equation*}
U(0):=\Phi(0,R;\lambda,B)P^s(R;\lambda,B)U^s_0 = \Phi(0,-R;\lambda,B)P^u(-R;\lambda,B)U^u_0
\end{equation*}
decays exponentially as $x\to +\infty$ as well as as $x\to-\infty$.

It follows that $\lambda$ is an eigenvalue of the perturbed operator $\mathcal{L}_B$ and the corresponding eigenfunction is the first component of $U(0)$.

Conversely, if $\lambda$ is an eigenvalue of the perturbed operator, then from ~\cref{lemma:expdecay} we have that ~\eqref{fullsystem} has a solution $U$ which decays exponentially as $|x| \to \infty$. 
Take 
\begin{equation*}
  \begin{aligned}
    U_0^s&:= P^s(R;\lambda_0,0) U(R), \\
    U_0^u&:= P^u(-R;\lambda_0,0) U(-R).
  \end{aligned}
\end{equation*}
Then $U_0^s$ and $U_0^u$ belong to $E^s_+$ and $E^u_-$ respectively, see ~\eqref{StableUnstableSubspaces}.
By the fact that $P^s(R;\lambda,B)P^s(R;\lambda_0,0)=P^s(R;\lambda,B)$ (see \cite[p.34]{Coppel78}) 
and by (iii) of Definition \ref{D:expdich} we obtain
\begin{equation*}
\Phi(0,R;\lambda,B)P^s(R;\lambda,B)U^s_0 = U(0) = \Phi(0,-R;\lambda,B)P^u(-R,\lambda,B)U^u_0
\end{equation*}
and thus equation ~\eqref{iotafunction} holds.
\end{proof}
Now let us focus on solving ~\eqref{iotafunction}. First note that for any $(U^s_0,U^u_0) \in E^s_+ \times E^u_- $ we have
\begin{equation*}
 \iota(U^s_0,U^u_0;\lambda_0,0) =  \Phi(0,R;\lambda_0,0)U^s_0 - \Phi(0,-R;\lambda_0,0)U^u_0.
\end{equation*}
Hence, $\Ran \iota(\cdot,\cdot, \lambda_0,0)= \Phi(0,R;\lambda_0,0) E^s_+ +\Phi(0,-R;\lambda_0,0) E^u_- $.

Next, let us focus on the codimension of $E^s_+ + E^u_- $.
\begin{lemma} \label{lemmacodim}
We have 
\begin{gather*}
\codim(\Phi(0,R;\lambda_0,0) E^s_+ + \Phi(0,-R;\lambda_0,0)E^u_-) = 2m + 1.
\end{gather*}
\end{lemma}

\begin{proof}
We first observe that the number of real negative and positive  eigenvalues of $M_\infty(\lambda_0)$ is $2(n-m)$, since $\dim X^s = n-m = \dim X^u$. Also, let us observe that from \cite[p.34]{Coppel78}, we have that $\dim E^s_+ = \dim X^s$ and $\dim E^u_- = \dim X^u$.
Therefore,
\begin{gather*}
\dim(E^s_+ + E^u_-) = \dim(E^s_+)+\dim(E^u_-)-\dim(E^s_+\cap E^u_-)= 2(n-m)-1.
\end{gather*}
Thus,
\begin{gather*}
\codim(\Phi(0,R;\lambda_0,0)E^s_+ +\Phi(0,-R;\lambda_0,0) E^u_-) = 2n-2n+2m+1 = 2m +1 = \dim(X^c) + 1.
\end{gather*}
\end{proof}
Let $Q$ be a projection in $\mathbb R^{2n}$ onto $\Ran\iota(\cdot,\cdot;\lambda_0, 0) =  \Phi(0,R;\lambda_0,0) E^s_+ +  \Phi(0,-R;\lambda_0,0) E^u_-$. 
Then ~\eqref{iotafunction} can be rewritten in the equivalent form
\begin{align}\label{iotasystem}
\begin{aligned}
Q\iota(U^s_0,U^u_0;\lambda,B) &= 0,\\
(I-Q)\iota(U^s_0,U^u_0;\lambda,B) &= 0.
\end{aligned}
\end{align}
~\cref{lemmacodim} implies that $\dim(\ker Q) = \codim(\Phi(0,R;\lambda_0,0)E^s_+ +\Phi(0,-R;\lambda_0,0) E^u_-) = 2m+1$.

We will start by solving the first equation of \eqref{iotasystem} using the implicit function theorem. In order to find a unique solution, we introduce an extra condition, which fixes one solution among the infinitely many in the one-dimensional subspace of solutions to this equation.
\begin{lemma} \label{lma8}
Let $D$ be a subspace of $E^s_+\times E^u_-$ such that $\operatorname{\rm span} \{(U_*(R),U_*(-R))\} + D = E^s_+\times E^u_-$ and $D\cap \operatorname{\rm span} \{(U^*(R),U^*(-R))\}=\{(0,0)\}$. Then for $(\lambda,B)$ close to $(\lambda_0,0)$, the first equation of \eqref{iotasystem} has a unique solution
\begin{equation*}
(U^s_0,U^u_0) = (U^s_0(\lambda,B),U^u_0(\lambda,B))
\end{equation*}
such that $(U^s_0,U^u_0)-(U_*(0),U_*(0))\in D$.
\end{lemma}
\begin{proof}
By Lemma ~\ref{smoothevolutionop}, $\iota:E_+^s\times E_-^u\times \mathbb R\times X_\beta \to X$ is smooth with respect to all variables, and so the same holds for $Q\iota$.
    
Note that $Q\iota$ is linear with respect to the first two variables, and that, by Lemma \ref{lemmacodim}, together with the assumption that the eigenvalue $\lambda_0$ is simple,
$\ker Q\iota(\cdot,\cdot;\lambda_0,0) = \operatorname{\rm span}\{(U_*(R),U_*(-R))\}$. Hence
 if $(U_0^s,U_0^u)-(U_*(R),U_*(-R))\in D$, then $Q\iota(U_0^s,U_0^u;\lambda_0,0)=0$ if and only if $(U_0^s,U_0^u)=(U_*(R),U_*(-R))$. 

By taking the restriction of $Q\iota$ to $(D+\{(U_*(R),U_*(-R))\})\times \mathbb R\times X_\beta$, we consider $Q\iota:(D+\{(U_*(0),U_*(0))\})\times \mathbb R\times X_\beta \to \Ran Q$, and then $Q\iota$ with this smaller domain is smooth too. By the above discussion, it follows that $Q\iota(\dot,\cdot;\lambda_0,0)$ with this domain is injective. 

By the definition of $Q$ and $\iota$, 
clearly $Q\iota(\cdot,\cdot;\lambda_0,0):E_+^s\times E_-^u \to \Ran Q$ is surjective. Since $E^s_+\times E^u_-=D\oplus \operatorname{\rm span} \{(U_*(R),U_*(-R))\}$, and 
$\operatorname{\rm span} \{(U_*(R),U_*(-R))\} = \ker \iota(\cdot,\cdot;\lambda_0,0)$, 
it follows that $Q \iota(\cdot,\cdot;\lambda_0,0)$ is surjective also as a function from the smaller domain $D+\{(U_*(R),U_*(-R))\}$.
    
Hence, by the implicit function theorem, the claim follows.
\end{proof}
By the integral formula derived in the proof of the roughness theorem \cite[p.30]{Coppel78}, we have the following formula for the first term of the right hand side of \eqref{iotadef},
\begin{equation*}
  \begin{aligned}
    \Phi(0,R;\lambda,B) & P^s(R;\lambda,B) = \Phi(0,R;\lambda_0,0) P^s(R;\lambda_0,0) \\
    &- \int_0^R \Phi(0,\xi;\lambda_0,0) P^s(\xi;\lambda_0,0) N(\xi;\lambda,B)\Phi(\xi,R;\lambda,B)P^s(R;\lambda,B)\, d\xi \\
    &- \int_0^\infty \Phi^u(0,\xi;\lambda_0,0) N(\xi;\lambda,B) \Phi(\xi,R;\lambda,B)P^s(R;\lambda,B)\, d\xi,
  \end{aligned}
\end{equation*}
where 
\begin{equation*}
    N(\xi;\lambda,B)= \begin{bmatrix}
       0 & 0\\
       B(\xi)-(\lambda-\lambda_0) I & 0
\end{bmatrix}.
\end{equation*}
Similarly, for the second term of \eqref{iotadef}, we have
\begin{equation*}
    \begin{aligned}
      \Phi(0,-R;\lambda,B) & P^u(-R;\lambda,B) = \Phi(0,-R;\lambda_0,0) P^u(-R;\lambda_0,0) \\
    &+ \int_{-R}^0 \Phi(0,\xi;\lambda_0,0) P^u(\xi;\lambda_0,0) N(\xi;\lambda,B)\Phi(\xi,-R;\lambda,B)P^u(-R;\lambda,B)\, d\xi \\
    &+ \int_{-\infty}^0 \Phi^s(0,\xi;\lambda_0,0)N(\xi;\lambda,B) \Phi(\xi,-R;\lambda,B)P^u(-R;\lambda,B)\, d\xi.
    \end{aligned}
\end{equation*}
Combining the last two expressions with \eqref{iotadef}, we obtain the formula
\begin{equation*}
    \begin{aligned}
        \iota(U_0^s,U_0^u;\lambda,B) &=  \Phi(0,R;\lambda_0,0) U_0^s - \Phi(0,-R;\lambda_0,0) U_0^u \\
        &-\int_0^R \Phi(0,\xi;\lambda_0,0) P^s(\xi;\lambda_0,0) N(\xi;\lambda,B)\Phi(\xi,R;\lambda,B) P^s(R;\lambda,B) U_0^s\, d\xi \\
    &- \int_{-R}^0 \Phi(0,\xi;\lambda_0,0) P^u(\xi;\lambda_0,0) N(\xi;\lambda,B) \Phi(\xi,-R;\lambda,B)P^u(-R;\lambda,B) U_0^u\, d\xi \\
    &- \int_0^\infty \Phi(0,\xi;\lambda_0,0)(I-P^s(\xi;\lambda_0,0))N(\xi;\lambda,B) \Phi(\xi,R;\lambda,B)P^s(R;\lambda,B) U_0^s\, d\xi \\
    &- \int_{-\infty}^0  \Phi(0,\xi;\lambda_0,0)(I-P^u(\xi;\lambda_0,0))N(\xi;\lambda,B) \Phi(\xi,-R;\lambda,B)P^u(-R;\lambda,B)U_0^u\, d\xi.
    \end{aligned}
\end{equation*}
In order to solve the second equation of ~\eqref{iotasystem} we define $F:\mathbb{R} \times X_\beta \to \ker Q$ by 
\begin{equation}\label{eqF}
\begin{aligned}
F(\lambda,B) &=  \iota(U^s_0(\lambda,B),U^u_0(\lambda,B);\lambda,B)= (I-Q)\iota(U^s_0(\lambda,B),U^u_0(\lambda,B)) \\
&= 
-\int_0^R (I-Q) \Phi(0,\xi;\lambda_0,0) P^s(\xi;\lambda_0,0) N(\xi;\lambda,B) \Phi(\xi,R;\lambda,B) P^s(R;\lambda,B)U_0^s(\lambda,B)\, d\xi \\
    &- \int_{-R}^0 (I-Q) \Phi(0,\xi;\lambda_0,0) P^u(\xi;\lambda_0,0) N(\xi;\lambda,B) \Phi(\xi,-R;\lambda,B)P^u(-R;\lambda,B)U_0^u(\lambda,B)\, d\xi \\
    &- \int_0^\infty (I-Q) \Phi(0,\xi;\lambda_0,0)(I-P^s(\xi;\lambda_0,0))N(\xi;\lambda,B) \Phi(\xi,R;\lambda,B)P^s(R;\lambda,B) U_0^s(\lambda,B)\, d\xi \\
    &- \int_{-\infty}^0 (I-Q) \Phi(0,\xi;\lambda_0,0)(I-P^u(\xi;\lambda_0,0))N(\xi;\lambda,B) \Phi(\xi,R;\lambda,B)P^u(-R;\lambda,B)U_0^u(\lambda,B)\, d\xi.
\end{aligned}
\end{equation}
We are going to solve the equation $F(\lambda,B)=0$ since it is equivalent to solving ~\eqref{iotasystem}. For that, we define the adjoint equation for $\lambda=\lambda_0$ and $B=0$

\begin{equation}\label{adjointequation}
    W'=-(M_\infty(\lambda_0)+L(x;0))^*W.
\end{equation}

\noindent This system has exponential dichotomies on $\mathbb{R}_+$ and $\mathbb{R}_-$ denoted by $\Psi^s(x,x_0)$, $\Psi^{cu}(x,x_0)$ and $\Psi^u(x,x_0)$, $\Psi^{cs}(x,x_0)$ respectively. Furthermore, the dichotomies of the unperturbed system ~\eqref{fullsystem} with $\lambda=\lambda_0$ and $B=0$, and the ones of the adjoint system ~\eqref{adjointequation} are related in the following way
\begin{equation*}
    \begin{aligned}
    \Psi^s(x,x_0) = \Phi^{cu}(x_0,x)^*, \quad \Psi^{cu}(x,x_0) = \Phi^s(x_0,x)^*,\\
    \Psi^{cs}(x,x_0) = \Phi^{u}(x_0,x)^*, \quad \Psi^u(x,x_0) = \Phi^{cs}(x_0,x)^*.
    \end{aligned}
\end{equation*}
Let us note that $U^\bot_* := (-\vec{u}'_*,\vec{u}_*)^T$ solves the adjoint system ~\eqref{adjointequation} and it decays exponentially as $|x| \to \infty$. Also, for $U^s \in E^s_+$, $U^u \in E^u_-$ we have
\begin{align*}
    \begin{aligned}
      \frac{d}{dx} \langle U^\bot_*(x),\Phi^s(x,0)U^s \rangle = \frac{d}{dx} \langle \Psi^s(x,0)U^\bot_*(0),\Phi^s(x,0))U^s \rangle = 0, \\
      \frac{d}{dx} \langle U^\bot_*(x),\Phi^u(x,0)U^u \rangle = \frac{d}{dx} \langle \Psi^u(x,0)U^\bot_*(0),\Phi^u(x,0)U^u \rangle = 0
    \end{aligned}
\end{align*}
This is easy to check if we apply the product rule and insert the respective differential equation in both cases. Hence $\langle U^\bot_*(x),\Phi^s(x,0)U^s \rangle$ and $\langle U^\bot_*(x),\Phi^u(x,0)U^u \rangle$ are both constant,
and since  $\Phi^s(x,0)U^s\to 0$ as $x\to \infty$ and $\Phi^u(x,0)U^u\to 0$ as $x\to -\infty$, while
$U^\bot_*(x)$ is bounded on $\mathbb R$,
it then follows that $\langle U^\bot_*,U^s+U^u \rangle = 0$. Therefore, for any $U \in X$ we have
\begin{equation} \label{zeroIntegrals}
    \langle U^\bot_*(0),QU \rangle \; = 0 \quad \text{and so also} \quad \langle Q^*U^\bot_*(0),U \rangle \; = 0.
\end{equation}
This implies that $U^\bot_*(0) \in \ker Q^*$.

\begin{lemma} \label{lma9}
The equation $\langle U^\bot_*(0),F(\lambda,B) \rangle \; = 0$ defines a smooth function $\lambda(B)$ in a neighbourhood of $B=0$ such that $\lambda(0)=\lambda_0$. Also, for any $B \in X_\beta$
\begin{equation*}
    \lambda'(0)B = -\int_{-\infty}^{+\infty}(\vec{u}_*(x),B(x)\vec{u}_*(x))_{\mathbb{R}^n} \; dx.
\end{equation*}
\end{lemma}
\begin{proof}
We use the notation $F_*(\lambda,B):= \langle U_*^\perp(0),F(\lambda,B)\rangle$. Then we need to solve the equation $F_*(\lambda, B)=0$, where by ~\eqref{eqF} we get
\begin{equation*}
    \begin{aligned}
        F_*(\lambda,B)
        &=  -\int_0^R \langle U_*^\perp(0),\Phi(0,\xi;\lambda_0,0) P^s(\xi;\lambda_0,0) N(\xi;\lambda,B) \Phi(\xi,R;\lambda,B) P^s(R;\lambda,B)U_0^s(\lambda,B)\rangle \, d\xi \\
    &- \int_{-R}^0 \langle U_*^\perp(0),
    \Phi(0,\xi;\lambda_0,0) P^u(\xi;\lambda_0,0) N(\xi;\lambda,B) \Phi(\xi,-R;\lambda,B)P^u(-R;\lambda,B)U_0^u(\lambda,B)\rangle \, d\xi \\
    &- \int_0^\infty \langle U_*^\perp(0),\Phi(0,\xi;\lambda_0,0)(I-P^s(\xi;\lambda_0,0))N(\xi;\lambda,B) \Phi(\xi,R;\lambda,B)P^s(R;\lambda,B)U_0^s(\lambda,B)\rangle \, d\xi \\
    &- \int_{-\infty}^0 \langle U_*^\perp(0),\Phi(0,\xi;\lambda_0,0)(I-P^u(\xi;\lambda_0,0))N(\xi;\lambda,B) \Phi(\xi,R;\lambda,B)P^u(-R;\lambda,B)U_0^u(\lambda,B)\rangle \, d\xi. \\
    \end{aligned}
\end{equation*}

By ~\cref{smoothevolutionop}, $F_*$ is a smooth function of $\lambda \text{ and } B$ in a neighbourhood of $(\lambda_0,0)$ with $F_*(\lambda_0,0) = 0$ and 
\begin{equation*}¨
\begin{aligned}
   \frac{\partial F_*}{\partial \lambda}(\lambda_0,0) &=
    -\int_0^R \langle U_*^\perp(0),\Phi(0,\xi;\lambda_0,0) P^s(\xi;\lambda_0,0) \frac{\partial N}{\partial \lambda} \Phi(\xi,R;\lambda_0,0) P^s(R;\lambda_0,0)U_0^s(\lambda_0,0)\rangle \, d\xi \\
    &- \int_{-R}^0 \langle U_*^\perp(0),
    \Phi(0,\xi;\lambda_0,0) P^u(\xi;\lambda_0,0) \frac{\partial N}{\partial \lambda} \Phi(\xi,-R;\lambda_0,0)P^u(-R;\lambda_0,0)U_0^u(\lambda_0,0)\rangle \, d\xi \\
    &- \int_0^\infty \langle U_*^\perp(0),\Phi(0,\xi;\lambda_0,0)(I-P^s(\xi;\lambda_0,0)) \frac{\partial N}{\partial \lambda} \Phi(\xi,R;\lambda_0,0)P^s(R;\lambda_0,0)U_0^s(\lambda_0,0)\rangle \, d\xi \\
    &- \int_{-\infty}^0  \langle U_*^\perp(0),\Phi(0,\xi;\lambda_0,0)(I-P^u(\xi;\lambda_0,0)) \frac{\partial N}{\partial \lambda} \Phi(\xi,R;\lambda_0,0)P^u(-R;\lambda_0,0)U_0^u(\lambda_0,0) \rangle \, d\xi,
\end{aligned}
\end{equation*}
where
\begin{equation*}
    \frac{\partial N}{\partial \lambda}= \left(
    \begin{matrix}
       0 & 0 \\
       -I & 0 
    \end{matrix}
    \right).
\end{equation*}
The first two integrals are zero because of \eqref{zeroIntegrals} and so
\begin{equation*}
   \begin{aligned}
     \frac{\partial F_*}{\partial \lambda}(\lambda_0,0) &=
     -\int_0^\infty \langle U_*^\perp(\xi), \frac{\partial N}{\partial \lambda} U_*(0)\rangle \, d\xi + \int_{-\infty}^0 \langle U_*^\perp(\xi), \frac{\partial N}{\partial \lambda} U_*(0)\rangle \, d\xi \\
     &= \int_{-\infty}^\infty \|\vec{u}_*(\xi)\|^2\, d\xi = 1
     \end{aligned}
\end{equation*}
since we assumed that $\vec{u}_*$ is normalized.

\noindent Therefore we can now apply the implicit function theorem to solve for $\lambda$. Then $\lambda$ is a function of $B$ in a neighbourhood of $B=0$ with $\lambda(0)=\lambda_0$.

For the second part of the proof, by differentiating the function $F_*(\lambda(B),B)=0$ and evaluating at $B=0$ we derive the desired formula. 
\end{proof}
We proceed with the following lemma which is an application of the calculus of variations lemma and it will be used in the final step of the proof of the main theorem.

\begin{lemma} \label{lma10}
Let $\vec{v}(\xi)$ be continuous function $\mathbb{R} \to \mathbb{R}^n$. If for every $B \in X_\beta$, 
\begin{equation*}
\int_{-\infty}^{+\infty} (\vec{v}(\xi),B(\xi)\vec{u}_*(\xi) )d\xi=0
\end{equation*}
then $\vec{v}=0$.
\end{lemma}

\begin{proof}
First, we choose specific $j$ and $k$ and let the entry $b_{jk}(\xi)=b_{kj}(\xi) \neq 0$. We assume that $b_{jk}$ is smooth and with compact support and let all other entries in $B$ be equal to zero. Then 
\begin{equation*}
    \int_{-\infty}^{+\infty}v_j(\xi)b_{jk}(\xi)u_{*k}(\xi)+v_k(\xi)b_{jk}(\xi)u_{*j}(\xi)d\xi = \int_{-\infty}^{+\infty} b_{jk}(\xi)(v_j(\xi)u_{*k}(\xi)+v_k(\xi)u_{*j}(\xi)) d\xi=0.
\end{equation*}
By using the variational calculus fundamental lemma since $b_{jk}$ was arbitrary we get 
\begin{equation} \label{vectoreqlemma5.5}
v_j(\xi)u_{*k}(\xi)+v_k(\xi)u_{*j}(\xi) = 0.
\end{equation}

In particular, if $k=j$ we have $v_j(\xi)u_{*j}(\xi) = 0$. This gives us immediately that $v_j(\xi)=0$ for all $\xi$ such that $u_{*j}(\xi) \neq 0$. 

Let $\xi_0$ be such that $u_{*j}(\xi_0)=0$. Now we have two cases:
\begin{enumerate} [(i)]
    \item  If there is a sequence $\xi_l \to \xi_0$ such that  $u_{*j}(\xi_l) \neq 0$, then by the above $v_j(\xi_l)=0$ for all $l$. Then by continuity we will have that $v_j(\xi_0)=0$.
    \item  There is an interval $I \ni \xi_0$ such that $u_{*j}(\xi)=0$ for all $\xi\in I$. Then by ~\eqref{vectoreqlemma5.5} $v_j(\xi)u_{*k}(\xi) = 0$ for all $k$ and all $\xi\in I$. If for some $k$, $u_{*k}(\xi_0) \neq 0$ then $v_j(\xi_0)=0$. Otherwise, $u_{*k}(\xi_0)= 0$ for all $k$. Then again we have two cases, either there exists a sequence $\xi_l \to \xi_0$ such that for some $k$, $u_{*k}(\xi_l) \neq 0$, in which case $v_j(\xi_0)=0$ by the argument in (i), or for all $k$ there exists an interval $J \ni \xi_0$ such that $u_{*k}(\xi)=0$ for all $\xi\in J$, i.e. $\vec{u_*}(\xi)=0$ for all $\xi\in J$. But since $\vec{u}_*$ solves the equation it has to be zero everywhere by the uniqueness of solutions. This is a contradiction because $\vec{u}_*$ is an eigenfunction. This means that $v_j(\xi_0)=0$.
\end{enumerate}
This shows that $v_j(\xi)=0$ for all $\xi \in \mathbb{R}$. Since $j\in \{1,\dots, n\}$ was arbitrary, it follows that $\vec{v}$ is identically zero.
\end{proof}

We may now proceed with the proof of  ~\cref{thm1}.

\begin{proof}[Proof of ~\cref{thm1}]\label{proof:main}
Since $\ker Q$ is $2m+1$-dimensional by Lemma \ref{lemmacodim}, and we eliminated one condition in Lemma \ref{lma9}, there are now $2m$ conditions left to verify. The adjoint projection $Q^*$ has a $(2m+1)$-dimensional kernel, just like $Q$ itself.  We have seen that $U_*^\perp(0)\in \ker Q^*$. Next, we define $W_k(0)\in \mathbb R^{2n}$, $k=1,\dots, 2m$ such that 
$\{W_k(0);\; k=1,\dots, 2m\} \cup \{U_*^\perp(0)\}$ is a basis for $\ker Q^*$. For $k=1,\dots 2m$, let $W_k$ be the solution of the adjoint unperturbed system with initial value $W_k(0)$.

For $k=1,\dots 2m$, we let
\begin{equation*}
    F_k(B)=\langle W_k(0),F(\lambda(B),B)\rangle,
\end{equation*}
where $F$ is defined in ~\eqref{eqF} and note that $F_k:X_\beta\to \mathbb R$ are smooth functions since $F$ is smooth by ~\cref{smoothevolutionop}.

If for some $B\in X_\beta$, $F_k(B)= 0$ for all $k=1,\dots,2m$, then 
$F(\lambda(B),B)=0$ as $\{W_k(0)\; k=1,\dots, 2m\} \cup \{U_*^\perp(0)\}$ is a basis for $\ker Q^*$. Clearly, the converse statement also holds. 

To prove the theorem, we show that there is a $2m$-dimensional manifold of perturbations $B$ defined by the equations $F_k(B)=0$ for $k=1,\dots,2m$.
With the notation
\begin{equation*}
    U(x,B) = 
    \begin{cases}
       \Phi(x,R,\lambda(B),B) U_0^s(\lambda(B),B) &\text{for }x\ge 0, \\
       \Phi(x,-R,\lambda(B),B) U_0^u(\lambda(B),B) &\text{for }x<0,
    \end{cases}
\end{equation*}
we have
\begin{equation*}
   \begin{aligned}
    F_k(B) &= 
-\int_0^R \langle W_k(0),\Phi(0,\xi;\lambda_0,0) P^s(\xi;\lambda_0,0) (B(\xi) + 
    \lambda(B)-\lambda_0)N U(\xi,B)\rangle \, d\xi \\
    &- \int_{-R}^0 \langle W_k(0), \Phi(0,\xi;\lambda_0,0) P^u(\xi;\lambda_0,0) (B(\xi)
    + \lambda(B)-\lambda_0)N U(\xi,B) \rangle \, d\xi \\
    &- \int_0^\infty \langle  W_k(0),\Phi(0,\xi;\lambda_0,0)(I-P^s(\xi;\lambda_0,0))(B(\xi) +  \lambda(B)-\lambda_0) N U(\xi,B) \rangle \, d\xi \\
    &- \int_{-\infty}^0 \langle W_k(0),\Phi(0,\xi;\lambda_0,0)(I-P^u(\xi;\lambda_0,0))(B(\xi) + \lambda(B)-\lambda_0) N U(\xi,B)\rangle \, d\xi.
\end{aligned}
\end{equation*}
Since $W_k(0)=(I-Q^*(0))W_k(0)$, the first terms above are $0$. Hence
\begin{equation*}
  \begin{aligned}
    F_k(B) &=  \int_{-\infty}^\infty \langle  W_k(\xi),N(\xi;\lambda,B) U(\xi,B) \rangle \, d\xi \\
    &= \int_{-\infty}^\infty (\vec{w}_k(\xi),(B(\xi)-(\lambda(B)-\lambda_0))\vec{u}(\xi;B))\, d\xi,
  \end{aligned}
\end{equation*}
where $W_k=(-\vec{w}_k',\vec{w}_k)^T$ and $\vec{u}(\xi,B)$ is the vector consisting of the first $n$ components of $U(\xi,B)$ and $(\cdot,\cdot)$ denotes the inner product in $\mathbb R^n$.

We claim that $F_k'(0)$, $k=1,\dots,2m$ are linearly independent. Indeed,
\begin{equation*}
  \begin{aligned}
    F_k'(0)B &= \int_{-\infty}^\infty (\vec{w}_k(\xi),(B(\xi)- \lambda'(0)B)\vec{u}_*(\xi))\, d\xi \\
    &=\int_{-\infty}^\infty (\vec{w}_k(\xi),B(\xi)\vec{u}_*(\xi))\, d\xi + \int_{-\infty}^\infty (\vec{u}_*(\xi),B(\xi)\vec{u}_*(\xi))\, d\xi \int_{-\infty}^\infty (\vec{w}_k(\xi),\vec{u}_*(\xi))\, d\xi
  \end{aligned}
\end{equation*}
by Lemma \ref{lma8}.
Let $\alpha_1,\dots,\alpha_{2m} \in \mathbb R$ be such that for every $B\in X_\beta$,
\begin{equation*}
    0 = \sum_{k=1}^{2m} \alpha_k F_k'(0) B.
\end{equation*}
After a rearrangement, we then have
\begin{equation*}
    \int_{-\infty}^\infty (\vec{w}(\xi) + \vec{\alpha}_* \vec{u}_*(\xi),B(\xi) \vec{u}_*(\xi))\, d\xi = 0,
\end{equation*}
where we have used the notation  
\begin{equation*}
\left\{
  \begin{aligned}
    \vec{w}(\xi) &= \sum_{k=1}^{2m} \alpha_k \vec{w_k}(\xi) \\
    \vec{\alpha}_* &= \int_{-\infty}^\infty (\vec{w}(\xi),\vec{u}_*(\xi))\, d\xi. 
  \end{aligned}
  \right.
\end{equation*}

Now we apply ~\cref{lma10} with $\vec{v}:= \vec{w}+\vec{\alpha}_*\vec{u}_*$ and conclude that \begin{equation*}
    \sum_{k=1}^{2m} \alpha_k \vec{w}_k(\xi)+\vec{\alpha}_*\vec{u}_*(\xi)=0
\end{equation*} 
for all $\xi\in\mathbb R$.
Then also 
\begin{equation*}
    \sum_{k=1}^{2m} \alpha_k \vec{w}_k'(\xi)+\vec{\alpha}_* \vec{u}_*'(\xi)=0
\end{equation*} 
and in particular for $\xi=0$ we obtain 
\begin{equation*}
    \sum_{k=1}^{2m} \alpha_k W_k(0) + \vec{\alpha}_* U_*^\perp(0) =0.
\end{equation*}
But $\{W_k(0); k=1,\dots, 2m\} \cup \{U^\bot_*(0)\}$ is a basis for $\ker Q^*$, therefore $\alpha_k=\vec{\alpha_*}=0$. Now it follows that $F'_k(0)$ are linearly independent. This concludes the proof of our main theorem.

\end{proof}

\section*{Acknowledgements}
The second author would like to thank Carina Geldhauser for useful comments and feedback.


\bibliography{reference}{}
\bibliographystyle{plain}

\end{document}